\documentclass[12pt]{amsart}

\usepackage{mathrsfs}

\newtheorem{theorem}{Theorem}[section]
\newtheorem{lemma}[theorem]{Lemma}
\newtheorem{corollary}[theorem]{Corollary}

\newcommand{\cO}{{\mathcal O}}

\newcommand{\cT}{{\mathcal T}}
\newcommand{\sT}{{\mathscr T}}

\newcommand{\nni}{\textsc{nni}}
\newcommand{\Nni}{\textsc{Nni}}
\newcommand{\spr}{\textsc{spr}}
\newcommand{\Spr}{\textsc{Spr}}
\newcommand{\tbr}{\textsc{tbr}}

\newcommand{\nhd}{neighbourhood}

\title[On the neighbourhoods of trees]{On the neighbourhoods of trees}

\author{Peter J. Humphries}
\address{Department of Mathematics and Statistics, University of Canterbury, Christchurch, New Zealand}
\email{pjhumphries@gmail.com}

\author{Taoyang Wu}
\address{School of Computing Sciences, University of East Anglia, Norwich, United Kingdom}
\email{taoyang.wu@gmail.com}

\thanks{The authors thank the Isaac Newton Institute for Mathematical Sciences, Cambridge, where this work began. Peter Humphries was supported by the New Zealand Marsden Fund. Taoyang Wu thanks Queen Mary, University of London, where he undertook some of this research. He was also supported in part by the Engineering and Physical Sciences Research Council (EP/D068800/1). A preliminary version of this paper is included in the authors' respective Ph.D. dissertations. Finally, we thank Dr. Arnau Mir Torres and two anonymous referees for their valuable comments.}



\date{\today}

\begin{document}

\begin{abstract}
Tree rearrangement operations typically induce a metric on the space of phylogenetic trees. One important property of these metrics is the size of the {\nhd}, that is, the number of trees exactly one operation from a given tree. We present an expression for the size of the {\tbr} (tree bisection and reconnection) neighbourhood, thus answering a question first posed in \cite{all01}.
\end{abstract}

\maketitle

\section{Introduction}
\label{intro}

Phylogenetic trees are a commonly used tool for representing the relationships between species in an evolutionary system, especially in evolutionary biology. A central task in the study of these trees is to determine which among a set of hypothesised trees gives the best explanation of empirical data. However, finding the trees that optimize some criterion is often computationally prohibitive because of the large number of trees to be checked. An approach that avoids this is a heuristic hill-climbing algorithm that searches tree space using tree rearrangement operations \cite{kat99,kub08}. That is, at each iteration the optimal tree within one rearrangement operation is chosen as the input for the next step, and the algorithm is thus guaranteed to find a local optimum.

 Loosely speaking, a tree rearrangement operation breaks a tree into two contiguous parts, and rejoins these parts to form a new tree. Among the three tree rearrangement operations of interest, namely {\nni} (nearest neighbour interchange), {\spr} (subtree prune and regraft) and {\tbr} (tree bisection and reconnection), each induces a distinct metric on the space of unrooted trees. Several properties of these metrics are important for understanding the efficiency of the algorithm outlined above. Our interest in this paper is in the size of the {\tbr} {\nhd}, that is, the number of trees that can be reached from a specified starting tree via a single {\tbr} operation.

A phylogenetic tree is an unrooted binary tree in the graph theoretic sense, with a unique label attached to every leaf, or vertex of degree one. We denote by $\sT_n$ the collection of all phylogenetic trees whose leaves are the set $\{1,\ldots,n\}$.

For a tree $\cT\in\sT_n$, where $n\geq 4$, Robinson \cite{rob71} showed that the {\nni} {\nhd} $N_{\nni}(\cT)$ has size exactly equal to $2n-6$, that is,
\begin{equation*}
%\label{nni:size}
|N_{\nni}(\cT)|=2n-6,
\end{equation*}
 while Allen and Steel \cite{all01} proved that
\begin{align*}
|N_{\spr}(\cT)|&=2(n-3)(2n-7),
\end{align*}
where $N_{\spr}(\cT)$ is the {\spr} neighbourhood of $\cT$. It was also demonstrated in \cite{all01} that the size of the {\tbr} {\nhd} is dependent on the shape of $\cT$. More recently, in~\cite{hum}, the bounds
\begin{align*}
cn^2\log n+O(n^2)&\leq|N_{\tbr}(\cT)|\leq\frac{2}{3}n^3-4n^2+\frac{16}{3}n+2
\end{align*}
were shown to hold for all $n\geq 4$, with the upper bound being met with equality if and only if $\cT$ is a caterpillar, that is, {   a phylogenetic tree in which every non-leaf vertex is adjacent to a leaf}.

The rest of this paper is divided into three sections. Section~\ref{background} contains the definitions required to follow the main content of the paper. In Section~\ref{exact}, we relate the number of possible rearrangement operations for $\cT$ to the size of the neighbourhood, and use this to reprove Allen and Steel's \cite{all01} result for the {\spr} neighbourhood and to obtain an expression for the {\tbr} neighbourhood dependent on the tree shape. In Section~\ref{extreme}, we characterise the trees that respectively maximise and minimise the size of $N_{\tbr}(\cT)$ for all binary tree spaces $\sT_n$. These characterisations are also extended to reprove the tight upper bound given in \cite{hum}, and to further prove an asymptotically tight lower bound.

\section{Definitions}
\label{background}

Before giving formal definitions of each rearrangement operation, we introduce some useful terminology. Given a tree $\cT$ and a subset $X$ of the leaf set of $\cT$, the restriction of $\cT$ to $X$, or $\cT|X$, is the minimal subtree of $\cT$ connecting the leaves in $X$, with all vertices of degree two supressed. A split $X|Y$ of a tree is a bipartition of the leaf set such that $\cT|X$ and $\cT|Y$ are vertex disjoint subtrees of $\cT$. Further, if $X'\subseteq X, Y'\subseteq Y$, then we call $X'|Y'$ a partial split of $\cT$. A split is trivial if one of its parts contains only one leaf. The set of all splits of $\cT$ is denoted by $\Sigma(\cT)$. If $\cT$ is a tree with the leaf set $Z$, then a cluster of $\cT$ is a set $X$ such that $X|(Z-X)\in\Sigma(\cT)$. If $|X|=2$, then we call $X$ a cherry.

{   A binary tree is a tree whose vertex degree is either one or three. Note that a binary tree with $n\geq 3$ leaves has $2n-2$ vertices in total and $2n-3$ edges, an observation that will be used throughout this paper. }

Although {\nni} was the point of departure for the study of these operations \cite{rob71}, we will first define {\tbr}, being the most general of the three. A {\tbr} operation on a binary phylogenetic tree $\cT$ involves deleting some edge $e$ from $\cT$ (bisection), and subsequently inserting a new edge $f$ so that the resulting tree $\cT'$ is distinct from $\cT$ (reconnection). Since we require $\cT'$ to be binary, it is necessary to subdivide an edge in one (in the case that the other component is an isolated labelled vertex) or both components created in the bisection stage before inserting the new edge. An example is given in Fig.~\ref{figtbrex}. We can transform $\cT_1$ into $\cT_2$ by first deleting the edge $e$ from $\cT_1$, and then adding the new edge $f$. To check that there has been no other change to the tree's structure, note that deleting $e$ from $\cT_1$ gives the same forest as deleting $f$ from $\cT_2$.
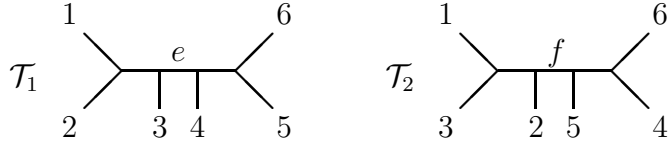
\begin{figure}[ht]
\begin{center}
\setlength{\unitlength}{5mm}
\begin{picture}(19,5)(0,0)
 \multiput(0,0)(10,0){2}{
 \thicklines
 \put(3,2){\line(1,1){1}}
 \put(3,4){\line(1,-1){1}}
 \put(4,3){\line(1,0){3}}
 \put(5,3){\line(0,-1){1}}
 \put(6,3){\line(0,-1){1}}
 \put(7,3){\line(1,1){1}}
 \put(7,3){\line(1,-1){1}}}
 \put(1,2.6){$\cT_1$}
 \put(2.4,4.25){$1$}
 \put(2.4,1.25){$2$}
 \put(4.8,1.25){$3$}
 \put(5.3,3.25){$e$}
 \put(5.8,1.25){$4$}
 \put(8.1,1.25){$5$}
 \put(8.1,4.25){$6$}
 \put(11,2.6){$\cT_2$}
 \put(12.4,4.25){$1$}
 \put(12.4,1.25){$3$}
 \put(14.8,1.25){$2$}
 \put(15.3,3.25){$f$}
 \put(15.8,1.25){$5$}
 \put(18.1,1.25){$4$}
 \put(18.1,4.25){$6$}
\end{picture}
\end{center}
\caption{Two trees $\cT_1,\cT_2\in\sT_6$ that are one {\tbr} operation apart.}
\label{figtbrex}
\end{figure}

For a binary tree $\cT$, we define the set $\cO_{\tbr}(\cT)$ to be all possible {\tbr} operations $\theta$ that can be applied to the tree $\cT$. An important point to note here is that for distinct $\theta_1,\theta_2\in\cO_{\tbr}(\cT)$, we may have $\theta_1(\cT)=\theta_2(\cT)$. The reason for this is that an operation $\theta\in\cO_{\tbr}(\cT)$ is not specified solely by the output tree $\theta(\cT)$, but also by the edge $e$ that is deleted from $\cT$ in the bisection stage of $\theta$.

Observe that for any two distinct trees $\cT,\cT'\in\sT_X$, there is a {\tbr} operation $\theta\in\cO_{\tbr}(\cT)$ for which $\theta(\cT)=\cT'$ if and only if there is some split $X_1|X_2\in\Sigma(\cT)\cap\Sigma(\cT')$ such that $\cT|X_i=\cT'|X_i$ for all $i\in\{1,2\}$. In this case, $X_1|X_2$ is the split induced by $\theta$. To demonstrate this, if the edges $e$ and $f$ have respectively been deleted and inserted in the {\tbr} operation that changes $\cT$ into $\cT'$, then the forest obtained by deleting $e$ from $\cT$ must be identical to the forest obtained by deleting $f$ from $\cT'$. This provides not only the common bipartition of the leaf set, but also the common subtrees induced by each part of this bipartition.

{\Spr} is a special case of {\tbr} in which there is less freedom at the reconnection stage. Let $\cT$ be a binary tree, and let $\theta\in\cO_{\tbr}(\cT)$ be a {\tbr} operation on $\cT$ in which the edge $e$ is deleted, and let $X_1|X_2$ be the split of $\cT$ induced by $e$. Then $\theta$ is an {\spr} operation for $\cT$ if and only if, without loss of generality, $\cT|(X_2\cup x_1)=\theta(\cT)|(X_2\cup x_1)$ for some $x_1\in X_1$. Moreover, if this holds then in fact the same property holds for all $x_1\in X_1$.

The significance of this condition is that one of the components formed in the bisection of $\cT$, in this case $\cT|X_2$, is treated as a rooted subtree, and is then regrafted so that this rooting is preserved with respect to the other component. We say that we have pruned $\cT|X_2$ from $\cT$, and regrafted it to form $\cT'$.

The previous example (refer to Fig.~\ref{figtbrex}) does not represent an {\spr} operation, since neither component obtained by deleting $e$ from $\cT_1$ can be regrafted to the other to form $\cT_2$. By making a subtle change, in particular by exchanging the labels $4$ and $5$ on $\cT_2$, we get a tree $\cT_3$ that can be obtained from $\cT_1$ by a single {\spr} operation. This example is shown in Fig.~\ref{figsprex}.
\begin{figure}[ht]
\begin{center}
\setlength{\unitlength}{5mm}
\begin{picture}(19,5)(0,0)
 \multiput(0,0)(10,0){2}{
 \thicklines
 \put(3,2){\line(1,1){1}}
 \put(3,4){\line(1,-1){1}}
 \put(4,3){\line(1,0){3}}
 \put(5,3){\line(0,-1){1}}
 \put(6,3){\line(0,-1){1}}
 \put(7,3){\line(1,1){1}}
 \put(7,3){\line(1,-1){1}}}
 \put(1,2.6){$\cT_1$}
 \put(2.4,4.25){$1$}
 \put(2.4,1.25){$2$}
 \put(4.8,1.25){$3$}
 \put(5.3,3.2){$e$}
 \put(5.8,1.25){$4$}
 \put(8.1,1.25){$5$}
 \put(8.1,4.25){$6$}
 \put(11,2.6){$\cT_3$}
 \put(12.4,4.25){$1$}
 \put(12.4,1.25){$3$}
 \put(14.8,1.25){$2$}
 \put(15.3,3.2){$f$}
 \put(15.8,1.25){$4$}
 \put(18.1,1.25){$5$}
 \put(18.1,4.25){$6$}
\end{picture}
\end{center}
\caption{Two trees $\cT_1,\cT_3\in\sT_6$ that are one {\spr} operation apart.}
\label{figsprex}
\end{figure}
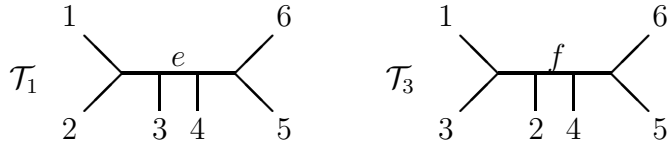

{\Nni} operations are {\tbr} operations in which the reconnection is still more restricted than for {\spr}. Let $\cT$ be a phylogenetic tree, and let $\theta\in\cO_{\tbr}(\cT)$ be an {\spr} operation in which $\cT|Y$ is pruned from $\cT$ and regrafted to form $\cT'=\theta(\cT)$. We say that $\theta$ is an {\nni} operation if and only if there is some cluster $Z\neq Y$ of $\cT$ such that we can form $\cT'$ from $\cT$ by swapping the subtrees $\cT|Y$ and $\cT|Z$. In this case, $\cT|Y$ and $\cT|Z$ can be seen as adjacent in some sense, as shown by the schematic diagram in Fig.~\ref{fignniex}. {   Note that $\cT'$ can be obtained from $\cT$ by four distinct} {\nni} {   operations, namely pruning one of the subtrees in $\{ \cT|X, \cT|Y, \cT|Z,\cT|W \}$ and regrafting it in an appropriate way. Indeed,  if $\theta$ is an} {\nni} {   operation for $\cT$, then there are precisely four distinct operations} $\theta'\in\cO_{\nni}(\cT)$ {   such that $\theta(\cT)=\theta'(\cT)$}.
%Although it may not be immediately obvious, the example in Fig.~\ref{figsprex} shows an {\nni} operation in which the leaves $2$ and $3$ have been swapped. Alternatively, the same outcome is reached by interchanging the subtrees labelled by $\{1\}$ and $\{4,5,6\}$ respectively. 
The possibility that two distinct operations can result in the same tree lies behind the main lemma (Lemma~\ref{lemnni}) in Section~\ref{exact}.
\begin{figure}[ht]
\begin{center}
\setlength{\unitlength}{5mm}
\begin{picture}(20,5)(0,0)
 \thicklines
 \multiput(0,0)(10,0){2}{
 \put(2,2.5){\line(0,1){1}}
 \put(2,2.5){\line(2,1){1}}
 \put(2,3.5){\line(2,-1){1}}
 \put(3,3){\line(1,0){4}}
 \multiput(4,3)(2,0){2}{
 \put(0,0){\line(0,-1){1}}
 \put(0,-1){\line(-1,-2){0.5}}
 \put(0,-1){\line(1,-2){0.5}}
 \put(-0.5,-2){\line(1,0){1}}}
 \put(7,3){\line(2,1){1}}
 \put(7,3){\line(2,-1){1}}
 \put(8,2.5){\line(0,1){1}}
 \put(1,2.75){$X$}
 \put(8.2,2.75){$W$}}
 \put(4.75,3.5){$\cT$}
 \put(3.7,0.25){$Y$}
 \put(5.7,0.25){$Z$}
 \put(14.75,3.5){$\cT'$}
 \put(13.7,0.25){$Z$}
 \put(15.7,0.25){$Y$}
\end{picture}
\end{center}
\caption{Two trees that are one {\nni} operation apart.}
\label{fignniex}
\end{figure}
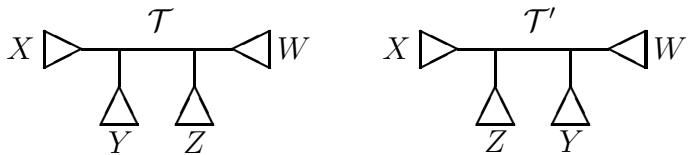

%\marginpar{Should we change the notation we use here? The upper and lower case thetas may be confusing to the reader.}

Extending our earlier notation for {\tbr} to both {\spr} and {\nni}, we have
\begin{align*}
\cO_{\nni}(\cT)&\subseteq\cO_{\spr}(\cT)\subseteq\cO_{\tbr}(\cT)
\end{align*}
for any tree $\cT$. 
%For each $\Theta\in\{\nni,\spr,\tbr\}$, the $\Theta$ {\nhd} of $\cT$ is the set
%\begin{align*}
%N_{\Theta}(\cT)&=\{\theta(\cT):\theta\in\cO_{\Theta}(\cT)\}.
%\end{align*}
%That is, $N_{\Theta}(\cT)$ is the set of all trees that are precisely one $\Theta$ rearrangement operation from $\cT$. 
 The {\tbr} {\nhd} of $\cT$ is the set
\begin{align*}
N_{\tbr}(\cT)&=\{\theta(\cT):\theta\in\cO_{\tbr}(\cT)\}.
\end{align*}
That is, $N_{\tbr}(\cT)$ is the set of all trees that are precisely one $\tbr$ rearrangement operation from $\cT$. The {\nni} neighbourhood $N_{\nni}(\cT)$ and the {\spr} neighbourhood $N_{\spr}(\cT)$ are defined similarly. Clearly, the elements in these neighbourhoods are dependent on the operation in question, and we have the corresponding nesting property as above. More explicitly,
\begin{align*}
N_{\nni}(\cT)&\subseteq N_{\spr}(\cT)\subseteq N_{\tbr}(\cT).
\end{align*}

\section{Neighbourhood Sizes}
\label{exact}

The approach used by Allen and Steel \cite{all01} to determine both the size of the {\spr} {\nhd} and the upper bound on the size of the {\tbr} {\nhd} was to count directly the number of trees that can be obtained from $\cT$ via a single operation. While this seems the most natural approach, there is a fundamental barrier to performing this enumeration that we alluded to briefly in {   Section~\ref{background}}. This is the fact that some operations in $\cO_{\tbr}(\cT)$ may be redundant. That is, there may be distinct elements $\theta_1,\theta_2\in\cO_{\tbr}(\cT)$ for which
\begin{align*}
\theta_1(\cT)&=\theta_2(\cT).
\end{align*}
This potentially leads to counting some trees in $N_{\tbr}(\cT)$ more than once. If we can determine precisely which operations in $\cO_{\tbr}(\cT)$ output the same tree, then we can relate the size of the {\tbr} {\nhd} to the number of operations on $\cT$.

It transpires, as the next lemma shows, that the only redundant {\tbr} operations are all {\nni} operations.

\begin{lemma}\label{lemnni}
Let $\theta,\theta'\in\cO_{\tbr}(\cT)$ be distinct {\tbr} operations. If $\theta(\cT)=\theta'(\cT)$, then $\theta\in \cO_{\nni}(\cT)$.
\end{lemma}

\begin{proof}
Suppose that $A|B$ is the split of $\cT$ induced by $\theta$, and that $A'|B'$ is the split induced by $\theta'$. Then $A|B\not= A'|B'$ as otherwise $\theta(\cT)$ must be distinct from $\theta'(\cT)$. Hence we may assume that $A\subset A'$, and hence also $B'\subset B$. Since $\cT|A'=\theta(\cT)|A'$, we have immediately that $\theta\in\cO_{\spr}(\cT)$. Let $A_0=A,A_1,\ldots,A_k=A'$ be clusters of $\cT$ such that
\begin{itemize}
\item[(i)]$A_i|B'$ is a partial split of $\cT$; and
\item[(ii)]$A_{i+1}$ is a minimal cluster of $\cT$ that contains $A_i$.
\end{itemize}

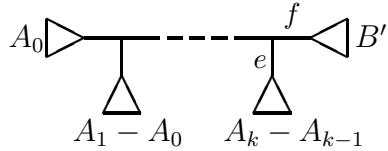
\begin{figure}[ht]
\begin{center}
\setlength{\unitlength}{5mm}
\begin{picture}(14,5)(0,0)
 \thicklines
 \put(3,2.5){\line(0,1){1}}
 \put(3,2.5){\line(2,1){1}}
 \put(3,3.5){\line(2,-1){1}}
 \put(4,3){\line(1,0){2}}
 \multiput(5,3)(4,0){2}{
 \put(0,0){\line(0,-1){1}}
 \put(0,-1){\line(-1,-2){0.5}}
 \put(0,-1){\line(1,-2){0.5}}
 \put(-0.5,-2){\line(1,0){1}}}
 \multiput(6,3)(0.6,0){3}{
 \put(0.2,0){\line(1,0){0.4}}}
 \put(8,3){\line(1,0){2}}
 \put(10,3){\line(2,1){1}}
 \put(10,3){\line(2,-1){1}}
 \put(11,2.5){\line(0,1){1}}
 \put(2,2.75){$A_0$}
 % \put(3.8,3.2){\small $v_0$}
  %\put(4.8,3.2){\small $v_1$}
 \put(3.7,0.25){$A_1-A_0$}
 \put(7.7,0.25){$A_k-A_{k-1}$}
 \put(9.3,3.4){{\small $f$}}
 \put(8.5,2.2){{\small $e$}}
  %\put(8.5,3.2){\small $v_k$}
   %\put(9,3.2){\small $v_{k+1}$}
 \put(11.2,2.75){$B'$}
\end{picture}
\end{center}
\caption{The tree $\cT$ in the proof of Lemma~\ref{lemnni}.}
\label{figlemma}
\end{figure}
The generic structure of $\cT$ is depicted in Fig.~\ref{figlemma},  where $f$ and $e$ are two edges and whose removal will result in the split $A'|B'$ and $(A_{k}-A_{k-1})|(X-(A_{k}-A_{k-1}))$, respectively. Now consider the operation $\theta$.
 If $k\geq 3$ then in order for $\cT|A'=\theta(\cT)|A'$ to hold, we must regraft the pruned subtree $\cT|A$ in the same place, but this implies $\cT=\theta(\cT)$, a contradiction.  If $k=2$, to ensure  $\cT|A'=\theta(\cT)|A'$, we must regraft $\cT|A$ to the edge $e$ or $f$. In other words,  $\theta(\cT)$ is obtained from $\cT$ by swapping either the subtrees $\cT|A$ and $\cT|B'$, or $\cT|A$ and $\cT|(A_2-A_1)$, from which it follows that $\theta$ is an {\nni} operation. Now it remains to establish the case $k=1$. To this end, we can further assume that $|A_1-A_0|>1$, because otherwise $\cT|A'=\theta(\cT)|A'$ implies the contradiction $\cT=\theta(\cT)$. Therefore the generic structure of $\cT$ in this case can be represented as in Fig.~\ref{fig:special}, where $C_1\cup C_2=A_1-A_0$. Using the constraint $\cT \not = \theta(\cT)$ and $\cT|A'=\theta(\cT)|A'$ again, we can assert that $\cT|A$ must be regrafted to either $e_1$ or $e_2$. This completes the proof as in both cases $\theta$ is an {\nni} operation.
 \end{proof}
 
 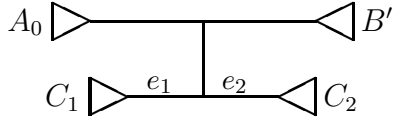
\begin{figure}[ht]
\begin{center}
\setlength{\unitlength}{5mm}
\begin{picture}(14,5)(0,0)
 \thicklines
 \multiput(0,0)(1,-2){2}{
 \put(3,2.5){\line(0,1){1}}
 \put(3,2.5){\line(2,1){1}}
 \put(3,3.5){\line(2,-1){1}}
}
\multiput(0,0)(-1,-2){2}{
 \put(10,3){\line(2,1){1}}
 \put(10,3){\line(2,-1){1}}
 \put(11,2.5){\line(0,1){1}}	 
 }
\put(4,3){\line(1,0){6}}
\put(5,1){\line(1,0){4}}
\put(7,3){\line(0,-1){2}}
 
 \put(1.8,2.75){$A_0$}
 \put(2.8,0.75){$C_1$}
 \put(5.5,1.2){{\small $e_1$}}
 \put(7.5,1.2){{\small $e_2$}}
 \put(11.2,2.75){$B'$}
 \put(10.2,0.75){$C_2$}
\end{picture}
\end{center}
\caption{The tree $\cT$ for the case $k=1$ in the proof of Lemma~\ref{lemnni}.}
\label{fig:special}
\end{figure}

As a consequence of Lemma~\ref{lemnni}, we can express the sizes of both the {\spr} and the {\tbr} {\nhd}s in terms of the number of each operation for a tree and the size of the {\nni} neighbourhood.

\begin{lemma}\label{lemdifference}
For $\cT\in\sT_n$, where $n\geq 4$, we have
\begin{align*}
|N_{\spr}(\cT)|&=|\cO_{\spr}(\cT)|-3|N_{\nni}(\cT)|,
\end{align*}
and
\begin{align*}
|N_{\tbr}(\cT)|&=|\cO_{\tbr}(\cT)|-3|N_{\nni}(\cT)|.
\end{align*}
\end{lemma}

\begin{proof}
This follows from Lemma~\ref{lemnni} and the observation {   in Section~\ref{background}} that, if $\theta$ is an {\nni} operation for $\cT$, then there are precisely four distinct operations $\theta'\in\cO_{\nni}(\cT)$ such that $\theta(\cT)=\theta'(\cT)$.
\end{proof}

This lemma forms the basis of the two key results for this section. Both the number of distinct {\spr} operations and the number of distinct {\tbr} operations for any given tree can be found relatively easily. We proceed with the {\spr} case first.

\begin{theorem}\label{thmspr}
For a tree $\cT\in\sT_n$ where $n\geq 4$, we have
\begin{align*}
|\cO_{\spr}(\cT)|&=4(n-2)(n-3).
\end{align*}
\end{theorem}

\begin{proof}
We consider two possible {\spr} operations on $\cT$, firstly
those that induce a trivial split on $\cT$, and secondly those that
induce a non-trivial split. In the first case, there are $n$ possible
leaves that can be pruned from $\cT$, and for each leaf $x$ there are $2n-6$ edges in
$\cT-x$ to which we can reconnect it so that the resulting tree is
different from $\cT$.

In the second case, suppose that the non-trivial split is $A|B$,
with $|A|=a$ and $|B|=b$. If we choose $\cT|A$ to be the pruned subtree,
then there are $2b-3$ edges to which we can regraft $\cT|A$.
However, one of these results in the same tree as we began with,
namely $\cT$. Thus there are $2b-4$ such distinct operations.
Similarly, if we choose $\cT|B$ as the pruned subtree, then there
are $2a-4$ possible {\spr} operations. Thus there are $2n-8$
distinct {\spr} operations for each of the $n-3$ non-trivial
splits of $\cT$. Hence
\begin{align*}
|\cO_{\spr}(\cT)|&=n(2n-6)+(n-3)(2n-8)\\
&=4(n-2)(n-3).
\end{align*}
\end{proof}

As a corollary to this theorem, we obtain the result of Allen and Steel's \cite{all01} for the size of the {\spr} {\nhd}. The proof is omitted, as it follows trivially from the size of {\nni} neighbourhood (see Section~\ref{intro}), Lemma~\ref{lemdifference} and Theorem~\ref{thmspr}.

\begin{corollary}[Theorem~2.1, \cite{all01}]\label{corspr}
For $\cT\in\sT_n$ where $n\geq 4$, we have
\begin{align*}
|N_{\spr}(\cT)|&=2(n-3)(2n-7).
\end{align*}
\end{corollary}

We require one further idea before tackling the {\tbr} problem. For a binary tree $\cT$, we define $\Gamma(\cT)$ by
\begin{align*}
\Gamma(\cT)&=\sum |A|\cdot|B|,
\end{align*}
where the sum is taken over all non-trivial splits $A|B$ of $\cT$. This quantity is closely related to the Wiener index which arose out of chemical graph theory \cite{dob01}.

\begin{theorem}\label{thmtbr}
For a tree $\cT\in\sT_n$ where $n\geq 4$, we have
\begin{align*}
|\cO_{\tbr}(\cT)|&=4\Gamma(\cT)-4(n-2)(n-3).
\end{align*}
\end{theorem}

\begin{proof}
We consider two possible {\tbr} operations on $\cT$, firstly
those that induce a trivial split on $\cT$, and secondly those that
induce a non-trivial split. The argument in the first case is identical to that given in the proof of Theorem~\ref{thmspr}, and gives $n(2n-6)$ distinct {\tbr} operations.

Now, let $A|B$ be some non-trivial split of $\cT$ induced by the
edge $e$. Then when we bisect $\cT$ by deleting $e$, there are
$2|A|-3$ edges in one component of the resulting forest and $2|B|-3$ edges in
the other. Hence, there are $(2|A|-3)(2|B|-3)$ ways to choose an
edge from each of $\cT|A$ and $\cT|B$. Precisely one of these
results in re-forming $\cT$. Hence, by taking a sum over all
non-trivial splits $A|B$ of $\cT$, we get
\begin{align*}
|\cO_{\tbr}(\cT)|&=n(2n-6)+\sum\left[(2|A|-3)(2|B|-3)-1\right]\\
&=4\Gamma(\cT)-4(n-2)(n-3).
\end{align*}
\end{proof}

This brings us to the following {   key result in this paper, which relates the size of the} {\tbr} {   neighbourhood of a phylogenetic tree to its shape, and provides an effective way to calculate this quantity}. 
% While the following corollary gives the size of the {\tbr} neighbourhood for $\cT$ in terms of $\Gamma(\cT)$, calculating this quantity is straightforward. 
Also, as we will see in the next section, Theorem~\ref{cortbr} gives us enough traction to characterise the trees that respectively maximise and minimise the size of the {\tbr} neighbourhood.

\begin{theorem}\label{cortbr}
For $\cT\in\sT_n$ where $n\geq 4$, we have
\begin{align*}
|N_{\rm TBR}(\cT)|&=4\Gamma(\cT)-(4n-2)(n-3).
\end{align*}
%where the sum is taken over all non-trivial splits $A|B$ of $\cT$.
\end{theorem}
\begin{proof}
This follows immediately from the size of the {\nni} neighbourhood, Lemma~\ref{lemdifference} and Theorem~\ref{thmtbr}.
\end{proof}

\section{Characterisations of the Extremal Cases}
\label{extreme}

Since the size of the {\tbr} {\nhd} for $\cT$ is dependent on both the number of leaves in $\cT$ and the shape of $\cT$, it makes sense to characterise which tree shapes give the extreme values for this size. As a consequence of Theorem~\ref{cortbr}, it suffices to determine which tree shapes maximise and minimise the size of $\Gamma(\cT)$ over all trees in $\sT_n$ for some $n$. We begin with the easier case, that is, finding the trees that maximise $\Gamma(\cT)$.

\begin{lemma}\label{lemmaximise}
Let $\cT\in\sT_n$ be a tree such that $\Gamma(\cT)\geq \Gamma(\cT')$ for all $\cT'\in\sT_n$. Then $\cT$ is a caterpillar.
\end{lemma}

\begin{proof}
Suppose that $\{x_1,x_2\}$ and $\{x_3,x_4\}$ are cherries of $\cT$, and let the sets $Y_1,\ldots,Y_k$ partition the remaining leaves so that $\cT$ can be represented as in Fig.~\ref{figcat}. 
\begin{figure}[ht]
\begin{center}
\setlength{\unitlength}{5mm}
\begin{picture}(14,6)(1,1)
 \thicklines
 \put(3,4){\line(1,0){4}}
 \put(4,4){\line(0,-1){1}}
 \multiput(6,4)(4,0){2}{
 \put(0,0){\line(0,-1){1}}
 \put(0,-1){\line(-1,-2){0.5}}
 \put(0,-1){\line(1,-2){0.5}}
 \put(-0.5,-2){\line(1,0){1}}}
 \multiput(7,4)(0.6,0){3}{
 \put(0.2,0){\line(1,0){0.4}}}
 \put(9,4){\line(1,0){4}}
 \put(12,4){\line(0,-1){1}} 
 \put(2,3.8){$x_1$}
 \put(3.7,2.25){$x_2$}
 \put(5.7,1.25){$Y_1$}
 \put(9.7,1.25){$Y_k$}
 \put(11.7,2.25){$x_3$}
 \put(13.2,3.8){$x_4$}
\end{picture}
\end{center}
\caption{The tree $\cT$ in the proof of Lemma~\ref{lemmaximise}.}
\label{figcat}
\end{figure}
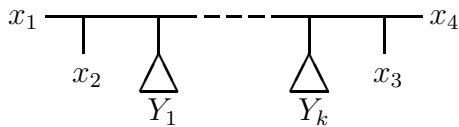

Setting $y_i=|Y_i|$, it will suffice to show that $y_i=1$ for all $i$.
{   Assuming otherwise, let  $i\in\{1,\ldots,k\}$ be  the smallest index such that $y_i>1$. Now}
 we form a second tree $\cT'$ by moving the subtree $\cT|Y_i$ to the position adjacent to $x_1$. The tree $\cT'$ is shown in Fig.~\ref{figalt}.
 
 Now, calculating the difference between $\Gamma(\cT)$ and $\Gamma(\cT')$, we find that
\begin{align*}
\Gamma(\cT)-\Gamma(\cT')&=\sum_{j=0}^{i-1}(j+2)(n-j-2)-\sum_{j=0}^{i-1}(y_i+j+1)(n-y_i-j-1)\\
%&=\sum_{j=0}^{i-1}[(j+2)(n-j-2)-(y_i+j+1)(n-y_j-j-1)] \\
&=\sum_{j=0}^{i-1}[(j+2)n-(j+2)^2-(y_i+j+1)n+(y_i+j+1)^2]\\
%&=\sum_{j=0}^{i-1}[(1-y_i)n+(y_i+2j+3)(y_i-1)]\\
&=(1-y_i)\sum_{j=0}^{i-1} (n-y_i-3-2j)\\
&=i(1-y_i)(n-y_i-i-2).
\end{align*}
Since $y_j\geq 1$ for all $j$, we have the inequality
%\begin{align*}
$y_i+(i-1)\leq n-4$,
%\end{align*}
from which $n-y_i-i-2$ is strictly positive. {   Together with the assumption $y_i>1$,
we conclude that $\Gamma(\cT)<\Gamma(\cT')$, a contradiction as required. Therefore $y_i=1$ indeed holds for all $i\in\{1,\ldots,k\}$}, and $\cT$ is a caterpillar.
 
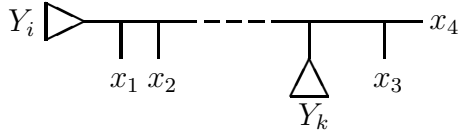
\begin{figure}[ht]
\begin{center}
\setlength{\unitlength}{5mm}
\begin{picture}(14,6)(1,1)
 \thicklines
 \put(3,3.5){\line(0,1){1}}
 \put(3,3.5){\line(2,1){1}}
 \put(3,4.5){\line(2,-1){1}}
 \put(4,4){\line(1,0){3}}
 \put(5,4){\line(0,-1){1}}
 \put(6,4){\line(0,-1){1}}
 \multiput(7,4)(0.6,0){3}{
 \put(0.2,0){\line(1,0){0.4}}}
 \put(9,4){\line(1,0){4}}
 \put(10,4){\line(0,-1){1}}
 \put(10,3){\line(-1,-2){0.5}}
 \put(10,3){\line(1,-2){0.5}}
 \put(9.5,2){\line(1,0){1}}
 \put(12,4){\line(0,-1){1}}
 \put(2,3.75){$Y_i$}
 \put(4.7,2.25){$x_1$}
 \put(5.7,2.25){$x_2$}
 \put(9.7,1.25){$Y_k$}
 \put(11.7,2.25){$x_3$}
 \put(13.2,3.8){$x_4$}
\end{picture}
\end{center}
\caption{The tree $\cT'$ in the proof of Lemma~\ref{lemmaximise}.}
\label{figalt}
\end{figure}
\end{proof}

Recall that the exact upper bound on the size of $N_{\tbr}(\cT)$ for a tree $\cT\in\sT_n$ was proven in \cite{hum} {   by induction on $n$}. Corollary~\ref{cormaximise} confirms this result {   using a different approach}.
\begin{corollary}[Theorem~2.1, \cite{hum}]\label{cormaximise}
The tree $\cT\in\sT_n$ maximises the size of the {\tbr} {\nhd} over $\sT_n$ if and only if $\cT$ is a caterpillar. Moreover, if $\cT$ is  a caterpillar then
\begin{align*}
|N_{\tbr}(\cT)|&=\frac{2}{3}n^3-4n^2+\frac{16}{3}n+2.
\end{align*}
\end{corollary}

\begin{proof}
The first part of the corollary follows from Lemma~\ref{lemmaximise}. To find the size of the neighbourhood, we apply Theorem~\ref{cortbr} from which we have
\begin{align*}
|N_{\rm TBR}(\cT)|&=4\Gamma(\cT)-(4n-2)(n-3)\\
&=4\sum_{i=2}^{n-2}i(n-i)-(4n-2)(n-3)\\
&=\frac{2}{3}n^3-4n^2+\frac{16}{3}n+2.
\end{align*}
\end{proof}

Characterising the trees that minimise the size of the {\tbr} neighbourhood relies heavily on Lemma~\ref{lemgamma}. Before proving this, we give an example of the simplest case of this lemma. Referring to Fig.~\ref{figsimple}, suppose that the sizes of the pendant subtrees labelled by $X_1,\ldots,X_4$ are $x_1,\ldots,x_4$ respectively.
If this tree has a minimal value for $\Gamma(\cT)$, then since $\Gamma(\cT)$ is the sum of $|A|\cdot |B|$ over all non-trivial splits $A|B$, we must have
\begin{align*}
(x_1+x_2)(x_3+x_4)&\leq \min\{(x_1+x_3)(x_2+x_4),(x_1+x_4)(x_2+x_3)\}.
\end{align*}
Assuming without loss of generality that $x_1$ is the smallest of the four quantities, it is easy to show that $x_2$ is the next smallest. Lemma~\ref{lemgamma} extends this observation to a more general result.
\begin{figure}[ht]
\begin{center}
\setlength{\unitlength}{5mm}
\begin{picture}(10,6)(1,1)
 \thicklines
 \put(3,3.5){\line(0,1){1}}
 \put(3,3.5){\line(2,1){1}}
 \put(3,4.5){\line(2,-1){1}}
 \put(4,4){\line(1,0){4}}
 \multiput(5,4)(2,0){2}{
 \put(0,0){\line(0,-1){1}}
 \put(0,-1){\line(-1,-2){0.5}}
 \put(0,-1){\line(1,-2){0.5}}
 \put(-0.5,-2){\line(1,0){1}}}
 \put(8,4){\line(2,1){1}}
 \put(8,4){\line(2,-1){1}}
 \put(9,3.5){\line(0,1){1}}
 \put(2,3.75){$X_1$}
 \put(4.7,1.25){$X_2$}
 \put(6.7,1.25){$X_3$}
 \put(9.2,3.75){$X_4$}
\end{picture}
\end{center}
\caption{A tree illustrating the simplest case of Lemma~\ref{lemgamma}.}
\label{figsimple}
\end{figure}
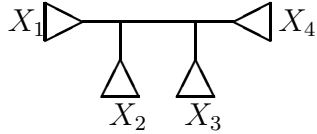

\begin{lemma}\label{lemgamma}
Let $X=\{1,\ldots,n\}$, and let $\cT\in\sT_n$ be such that $\Gamma(\cT)\leq \Gamma(\cT')$ for all $\cT'\in\sT_n$. Further, for some $k\geq 0$ let
$X_1,\ldots,X_4,Y_1,\ldots,Y_k$ partition $X$ such that the
following hold:
\begin{itemize}
\item[(i)]$X_i|(X-X_i)\in\Sigma(\cT)$ for all $i\in\{1,\ldots,4\}$;
\item[(ii)]$Y_i|(X-Y_i)\in\Sigma(\cT)$ for all $i\in\{1,\ldots,k\}$;
and
\item[(iii)]$A_i|(X-A_i)\in\Sigma(\cT)$ for all $i\in\{0,\ldots,k\}$, where $A_0=X_1\cup X_2,A_i=A_{i-1}\cup Y_i$.
\end{itemize}
Then without loss of generality we have $x_1\leq x_2\leq x_3\leq x_4$,
where $x_i=|X_i|$.
\end{lemma}

\begin{proof}
{   Swapping the subscripts of $X_i$ if necessary}, we can assume $x_1\leq x_2\leq x_3$.
Supposing that the lemma is false, we have $x_2>x_4$. Then either
$x_1=x_3$, and so $x_1\geq x_2\geq x_3\geq x_4$, contradicting our
assumption that the lemma is false, or $x_1<x_3$.

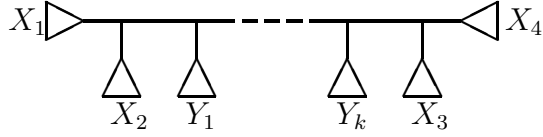
\begin{figure}[ht]
\begin{center}
\setlength{\unitlength}{5mm}
\begin{picture}(16,6)(1,1)
 \thicklines
 \put(3,3.5){\line(0,1){1}}
 \put(3,3.5){\line(2,1){1}}
 \put(3,4.5){\line(2,-1){1}}
 \put(4,4){\line(1,0){4}}
 \multiput(5,4)(2,0){2}{
 \put(0,0){\line(0,-1){1}}
 \put(0,-1){\line(-1,-2){0.5}}
 \put(0,-1){\line(1,-2){0.5}}
 \put(-0.5,-2){\line(1,0){1}}}
 \multiput(8,4)(0.6,0){3}{
 \put(0.2,0){\line(1,0){0.4}}}
 \put(10,4){\line(1,0){4}}
 \multiput(11,4)(2,0){2}{
 \put(0,0){\line(0,-1){1}}
 \put(0,-1){\line(-1,-2){0.5}}
 \put(0,-1){\line(1,-2){0.5}}
 \put(-0.5,-2){\line(1,0){1}}}
 \put(14,4){\line(2,1){1}}
 \put(14,4){\line(2,-1){1}}
 \put(15,3.5){\line(0,1){1}}
 \put(2,3.75){$X_1$}
 \put(4.7,1.25){$X_2$}
 \put(6.7,1.25){$Y_1$}
 \put(10.7,1.25){$Y_k$}
 \put(12.7,1.25){$X_3$}
 \put(15.2,3.75){$X_4$}
\end{picture}
\end{center}
\caption{The tree $\cT$ in Lemma~\ref{lemgamma}.}
\label{figgamma}
\end{figure}
Figure~\ref{figgamma} shows the general structure of a tree $\cT$ that satisfies the conditions of the lemma.
Let $\cT_1$ be the tree obtained from $\cT$ by swapping the subtrees labelled by $X_1$ and $X_3$, and let $\cT_2$ be similarly obtained by swapping the subtrees $\cT|X_2$ and $\cT|X_4$. Let $y_i=|Y_i|$, and $b_0=0,b_i=b_{i-1}+y_i$. Then we have
\begin{align*}
\Gamma(\cT)-\Gamma(\cT_1)&=\sum_{j=0}^k(x_1+x_2+b_j)(n-x_1-x_2-b_j)\\&\qquad-\sum_{j=0}^k(x_2+x_3+b_j)(n-x_2-x_3-b_j)\\
&=(x_3-x_1)\left[2\sum_{j=0}^k b_j-(k+1)(n-x_1-2x_2-x_3)\right].
\end{align*}
Since we assume that $\Gamma(\cT)\leq\Gamma(\cT_1)$, we get
\begin{align*}
\Gamma(\cT)-\Gamma(\cT_2)&=(x_4-x_2)\left[2\sum_{j=0}^k
b_j-(k+1)(n-2x_1-x_2-x_4)\right]\\
&>(x_4-x_2)\left[2\sum_{j=0}^k b_j-(k+1)(n-x_1-2x_2-x_3)\right]\\
&=\frac{x_4-x_2}{x_3-x_1}\left(\Gamma(\cT)-\Gamma(\cT_1)\right)\\
&\geq 0,
\end{align*}
contradicting the fact that $\Gamma(\cT)\leq \Gamma(\cT_2)$.
\end{proof}

Applying Lemma~\ref{lemgamma}, we can completely characterise those trees $\cT$ that minimise the size of $\Gamma(\cT)$, and therefore those trees that minimise the size of the {\tbr} {\nhd}.

\begin{lemma}\label{lemminimise}
Let $X=\{1,2,\cdots,n\}$ for some $n=\sum_{i=0}^k \alpha_i 2^i$, where
$\alpha_i\in\{0,1\}$ for $0\leq i<k$ and $\alpha_k=1$. Let
$\beta_j=\frac{1}{2^j}\sum_{i=j}^k \alpha_i 2^i$. Let $\cT\in\sT_n$
such that $\Gamma(\cT)\leq\Gamma(\cT')$ for all $\cT'\in\sT_n$. Then
for all $0\leq j\leq k-1$ there is a partition
$X_1,\ldots,X_{\beta_j}$ of $X$ into $\beta_j$ disjoint subsets such
that following properties hold:
\begin{itemize}
\item[(i)]$X_p|(X-X_p)\in\Sigma(\cT)$ for all $1\leq p\leq \beta_j$; and
\item[(ii)]$|X_p|=2^j$ for all $1\leq p<\beta_j$.
\end{itemize}
\end{lemma}

\begin{proof}
For $j=0$, this holds trivially. We assume that for some $0\leq
j<k-1$, the partition $X_1,\ldots,X_{\beta_j}$ of $X$ satisfies the
conditions of the lemma.

Suppose that for $1\leq p<q<\beta_j$, there is no set $Y$ that
contains either $X_p$ or $X_q$ such that $Y|(X-Y)\in\Sigma(\cT)$ and
$|Y|=2^{j+1}$. Then we can apply Lemma~\ref{lemgamma} to find a tree $\cT'$ for
which $\Gamma(\cT')<\Gamma(\cT)$. Hence, for $m$ such that
$2m<\beta_j$, there are disjoint subsets $X_1',\ldots,X_m'$ of $X$
such that $X_p'|(X-X_p')\in\Sigma(\cT)$ and $|X_p'|=2^{j+1}$.

There are two cases to consider. Suppose firstly that $2m=\beta_j-2$. Then there is some $1\leq p<
\beta_j$ such that $X_p$ is not contained in some $Y$, where
$Y|(X-Y)\in\Sigma(\cT)$ and $|Y|=2^{j+1}$. We can then use Lemma~\ref{lemgamma}
again to show that if $X_{\beta_{j+1}}'=X_p\cup X_{\beta_j}$, then
$X_{\beta_{j+1}}'|(X-X_{\beta_{j+1}}')\in\Sigma(\cT)$. Since
$m+1=\beta_{j+1}$, we have the required partition.

On the other hand, if $2m=\beta_j-1$ then we can use Lemma~\ref{lemgamma} to
show that there is some $1\leq p\leq m$ such that, if
$X_{\beta_{j+1}}'=X_{\beta_j}\cup X_p'$, then
$X_{\beta_{j+1}}'|(X-X_{\beta_{j+1}}')\in\Sigma(\cT)$. Again, this
gives the required partition, completing the induction.
\end{proof}

The question now is what these trees look like. In some sense, the trees that minimise the size of $\Gamma(\cT)$ are maximally balanced, although we must carefully define what we mean by this. The only sizes of $n$ for which an unrooted binary tree can be truly balanced, or perfect, are $n=2^k$ or $n=3\cdot 2^k$, where the tree is vertex-transitive with respect to the leaves and we have either two-fold symmetry about an interior edge of the tree or three-fold symmetry about an interior vertex. For values of $n$ other than those which admit a perfect tree, we necessarily lose the global property of leaf-transitivity.

{   A tree $\cT\in\sT_n$, where $3\cdot 2^k\leq n< 3\cdot 2^{k+1}$ for some $k\geq 0$, is called} complete if and only if
\begin{itemize}
\item[(i)]there is a cluster $Y$ of $\cT$ with $|Y|=2^{k+1}$; and
\item[(ii)]for all clusters $Y$ with $2\leq |Y|\leq 2^{k+1}$, there is a bipartition $Y_1,Y_2$ of $Y$ such that both of $Y_1,Y_2$ are clusters of $\cT$, and such that $|Y_1|=2^j$ and $2^{j-1}\leq|Y_2|<2^{j+1}$ for some $j$.
\end{itemize}
% Intuitively, for each cluster $Y$ with $|Y|$ being a power of 2}, the pendant subtree $\cT|Y$ has minimal depth, and one half of this pendant subtree is perfectly balanced.%\marginpar{Is this statement correct? I think we need to remove ``with $|Y|$ being a power of 2'', since in that case $\cT|Y$ is perfectly balanced. How about: ``Intuitively, for any cluster $Y$, the pendant subtree $\cT|Y$ has minimal depth and one half of this pendant subtree is perfectly balanced.''}
% Note that the notion of completeness can be generalized to trees with arbitrary vertex degrees (see~\cite{laszlo}).}  
Intuitively, for each cluster $Y$ with $|Y|$ being a power of 2, the pendant subtree $\cT|Y$  is perfectly balanced. For more details on complete trees and a generalization of completeness to trees with arbitrary vertex degrees, see~\cite{laszlo}.
The trees in Lemma~\ref{lemminimise} are precisely the complete trees in the space $\sT_n$, from which we obtain the next theorem. The proof is routine and omitted.

\begin{theorem}\label{thmminimise}
The tree $\cT\in\sT_n$ minimises the size of the {\tbr} {\nhd} over $\sT_n$ if and only if $\cT$ is complete.
\end{theorem}

Let us continue towards finding the size of the {\tbr} {\nhd} for complete trees. To this end, we introduce one additional notation. For each positive integer $m$, there exists a unique binary expansion $m=\sum_{i=0}^k \alpha'_i 2^i$, where
$\alpha'_i\in\{0,1\}$ for $0\leq i<k$ and $\alpha'_k=1$. Let $\tau(m)=1$ if $\alpha'_{k-1}=1$, and $\tau(m)=0$ otherwise. In particular, we have $\tau(2^k)=0$ for every $k$.  

\begin{lemma}\label{lemother}
Let $\cT\in\sT_n$ be a complete tree for some $n=\sum_{i=0}^k \alpha_i 2^i$, where
$\alpha_i\in\{0,1\}$ for $0\leq i<k$ and $\alpha_k=1$. Then:
\begin{align*}
\Gamma(\cT)&=
\sum_{j=1}^{k-1}\left[\left(\sum_{i=j}^k \alpha_i 2^i-2^j\right)\left(2n-\sum_{i=j}^k \alpha_i 2^i\right)+\alpha_{j-1}2^j(n-2^j)\right]\\
&\qquad+(\alpha_{k-1}-1)2^{k-1}(n-2^{k-1}).
\end{align*}
%if $\alpha_{k-1}=1$, and
%\begin{align*}
%\Gamma(\cT)&=\sum_{j=1}^{k-2}\left(\sum_{i=j}^k \alpha_i 2^i-2^j\right)\left(2n-\sum_{i=j}^k \alpha_i 2^i\right)+2^{k-1}(n-2^{k-1})
%\end{align*}
%if $\alpha_{k-1}=0$.
\end{lemma}

\begin{proof}
We use the proof of Lemma~\ref{lemminimise} to obtain this result. For each of the partitions $X_1,\ldots,X_{\beta_j}$, we take the sum of $|X_p|\cdot(n-|X_p|)$. 
Note that $X_{\beta_j}$ contains a cluster of size $2^{j}$ if and only if $\tau(|X_{\beta_j}|)=1$.

Consider a tree on $n$ leaves where $\alpha_{k-1}=1$ following the notation of Lemma~\ref{lemminimise}. This gives
\begin{align*}
\Gamma(\cT)&=\sum_{j=1}^{k-1}\left[\sum_{p=1}^{\beta_j}|X_p|\cdot(n-|X_p|)\right]+\tau(|X_{\beta_j}|)2^j(n-2^j)\\
&=\sum_{j=1}^{k-1}\left[2^j\Big(\beta_j-1+\tau(|X_{\beta_j}|)\Big)(n-2^j)+|X_{\beta_j}|\cdot(n-|X_{\beta_j}|)\right].
\end{align*}
We also have from Lemma~\ref{lemminimise} that
\begin{align*}
|X_{\beta_j}|&=n-2^j(\beta_j-1),
\end{align*}
and hence $\tau(|X_{\beta_j}|)=1$ if and only if $\alpha_{j-1}=1$.
Incorporating this into the above expression, we find
\begin{align*}
\Gamma(\cT)&=\sum_{j=1}^{k-1} \left[2^j(\beta_j-1)(2n-2^j\beta_j)+\alpha_{j-1}2^j(n-2^j)\right]\\
&=\sum_{j=1}^{k-1}\left[\left(\sum_{i=j}^k \alpha_i 2^i-2^j\right)\left(2n-\sum_{i=j}^k \alpha_i 2^i\right)+\alpha_{j-1}2^j(n-2^j)\right].
\end{align*}

In the case that $\alpha_{k-1}=0$, the partition $X_1,\ldots,X_{\beta_{k-1}}$ is a bipartition of the leaf set of $\cT$, and so we need only take the product $|X_1|\cdot(n-|X_1|)$ once in the sum above. In other words, we need to subtract $2^{k-1}(n-2^{k-1})$ from the formula.
%That is
%\begin{align*}
%\Gamma(\cT)&=\sum_{j=1}^{k-2}\left[\sum_{p=1}^{\beta_j}|X_p|\cdot(n-|X_p|)\right]+2^{k-1}(n-2^{k-1})\\
%&=\sum_{j=1}^{k-2}\left(\sum_{i=j}^k \alpha_i 2^i-2^j\right)\left(2n-\sum_{i=j}^k \alpha_i 2^i\right)+2^{k-1}(n-2^{k-1}).
%\end{align*}
\end{proof}

We conclude this section with two corollaries, the first of which gives an exact value for the size of the {\tbr} {\nhd} for perfect trees, and the second an asymptotic lower bound on the size of this neighbourhood for complete trees. Both proofs follow from Lemma~\ref{lemother} and Theorem~\ref{cortbr}.

\begin{corollary}\label{corperfect}
Let $\cT\in\sT_n$ be a perfect tree. Then
\begin{align*}
|N_{\tbr}(\cT)|&=n^2\left(4k-\frac{32}{3}\right)+22n-6
\end{align*}
if $n=3\cdot 2^{k-1}$ for some $k$, and
\begin{align*}
|N_{\tbr}(\cT)|&=n^2(4k-13)+22n-6
\end{align*}
if $n=2^k$ for some $k$.
\end{corollary}

%\marginpar{The statement of this corollary is for $n=3\cdot 2^{k-1}$ and $n=2^k$, but in the proof we have $n=3\cdot 2^k$ and $n=2^{k+1}$. I worked through the algebra in both cases and think that it is correct for $n=3\cdot 2^{k-1}$ and $n=2^k$. Please check this also so we can make it consistent.}

\begin{proof}
In the first case, where $n=3\cdot 2^{k-1}$, we have
\begin{align*}
\Gamma(\cT)&=\sum_{j=1}^{k-1} n(n-2^j)\\
&=n^2(k-1)-n(2^k-2)\\
&=n^2\left(k-\frac{5}{3}\right)+2n,
\end{align*}
and the result follows by applying Theorem~\ref{cortbr}. On the other hand, if $n=2^{k}$ then
\begin{align*}
\Gamma(\cT)&=\sum_{j=1}^{k-2}n(n-2^j)+\frac{n^2}{4}\\
&=n^2\left(k-\frac{7}{4}\right)-n(2^{k-1}-2)\\
&=n^2\left(k-\frac{9}{4}\right)+2n,
\end{align*}
and again applying Theorem~\ref{cortbr} gives the required result.
\end{proof}

\begin{corollary}\label{corminimise}
Let $\cT\in\sT_n$ be a complete tree. Then
\begin{align*}
|N_{\tbr}(\cT)|&=4n^2\lfloor\log_2 n\rfloor+O(n^2).
\end{align*}
\end{corollary}

\begin{proof}
The proof is similar in nature to that for 
%Corollary~\ref{corperfect}.
the previous corollary. 
If $3\cdot 2^{k-1}\leq n<2^{k+1}$ for some $k\geq 1$, then we have
\begin{align*}
\Gamma(\cT)&=\sum_{j=1}^{k-1}\left[\left(n-\sum_{i=0}^{j-1}\alpha_i 2^i-2^j\right)\left(n+\sum_{i=0}^{j-1}\alpha_i 2^i\right)+\alpha_{j-1}2^j(n-2^j)\right]
\\
&=n^2(k-1)-n(2^k-2)+\sum_{j=1}^{k-1} \alpha_{j-1}2^j(n-2^j)-\sum_{j=1}^{k-1}\left(\sum_{i=0}^{j-1}\alpha_i 2^i\right)^2.
\end{align*}
However, we can obtain a bound for the final term of this expression by assuming that $\alpha_i=1$ for all $i\in\{0,\ldots,k-2\}$, giving
\begin{align*}
\sum_{j=1}^{k-1}\left(\sum_{i=0}^{j-1}\alpha_i 2^i\right)^2&<\sum_{j=1}^{k-1}2^{2j}\\
&=\frac{2}{3}\left(2^{2k-1}-1\right)\\
&=O(n^2).
\end{align*}
Similarly, we have $\sum_{j=1}^{k-1} \alpha_{j-1}2^j(n-2^j)=O(n^2)$, as required.

The other case, where $2^k\leq n<3\cdot 2^{k-1}$, follows in a similar manner, and we complete the proof by applying Theorem~\ref{cortbr}.
\end{proof}

\end{document}